\documentclass[11pt,notitlepage,twoside,a4paper]{amsart}
 \usepackage{amsfonts}
\usepackage[cyr]{aeguill}
\usepackage{xspace}
\usepackage[polutonikogreek,frenchb]{babel}

\usepackage{amsmath,amssymb,enumerate}

\usepackage{epsfig,fancyhdr,color}

\usepackage{amssymb}
\usepackage{amsmath,amsthm}
\usepackage{latexsym}
\usepackage{amscd}
\usepackage{psfrag}
\usepackage{graphicx}
\usepackage[all]{xy}

\definecolor{NoteColor}{rgb}{1,0,0}

\renewcommand{\textsc}{\textcolor{red}}

\newtheorem{theorem}{\rm\bf Th\'eor\`eme}

\newtheorem{lemma}[theorem]{\rm\bf Lemme}

\newtheorem*{theorem 1}{\rm\bf Proposition 1}
\newtheorem*{theorem 2}{\rm\bf Proposition 2}

\theoremstyle{definition}

\theoremstyle{remark}
\newtheorem{remark}[theorem]{\rm\bf Remarque}

\def\interieur#1{\mathord{\mathop{\kern 0pt #1}\limits^\circ}}

\title[Hom\'eomorphismes et nombre d'intersection]{Hom\'eomorphismes et nombre d'intersection}

\author[Ken'ichi Ohshika et Athanase Papadopoulos]{Ken'ichi Ohshika et Athanase Papadopoulos}

 \address{A. Papadopoulos, Institut de Recherche Mathématique Avancée (Université de Strasbourg et CNRS),
7 rue René Descartes
67084 Strasbourg Cedex France,  papadop@math.unistra.fr ; K. Ohshika, 
Department of Mathematics, Graduate School of Science, Osaka University Toyonaka, Osaka 560-0043, Japan, 
ohshika@math.sci.osaka-u.ac.jp}

\date{\today}

\begin{document}
   
 \maketitle
 
 \begin{abstract}
 On démontre deux résultats de rigidité pour des groupes d'automorphismes de l'espace $\mathcal{ML}(S)$ des laminations géodésiques mesurées d'une surface fermée hyperbolique $S$ et de l'espace $\mathcal{PML}(S)$ des laminations géodésiques mesurées projectives de $S$. Ils concernent les automorphismes de  $\mathcal{ML}(S)$ préservant le nombre d'intersection géométrique entre laminations et les homéomorphismes de $\mathcal{PML}(S)$ préservant les ensembles de zéros de ces fonctions.

 \medskip
 
 \noindent {\sc abstract}. We prove two rigidity results for automorphism groups of the spaces $\mathcal{ML}(S)$ of measured laminations on a closed hyperbolic surface $S$ and  $\mathcal{PML}(S)$ of projective measured laminations on this surface. The results  concern the homeomorphisms of $\mathcal{ML}(S)$ that preserve the geometric intersection between laminations and the homeomorphisms of $\mathcal{PML}(S)$ that preserve the zero sets of these intersection functions.

 \medskip
 
 \noindent Classification AMS: 37E30, 57M99.
 
 \end{abstract}

Soit $S$ une surface fermée de genre $g\geq 2$, $\mathcal{ML}(S)$ l'espace des  laminations géodésiques mesurées de $S$ (pour une certaine structure hyperbolique) et $\mathcal{PML}(S)$ le projectifié de cet espace.  On désigne par $i$ le nombre d'intersection géométrique entre éléments de $\mathcal{ML}(S)$.

Dans cet article, on démontre les deux théorèmes de rigidité suivants :
\begin{theorem}
\label{ml1}
Soit $f: \mathcal{ML}(S)\to \mathcal{ML}(S)$ un homéomorphisme tel que pour tout $\lambda$ et $\mu$ dans $\mathcal{ML}(S)$, on a $i(\lambda, \mu)=i(f(\lambda), f(\mu))$.
Alors, $f$ est induit par un homéomorphisme de $S$.
\end{theorem}

\begin{theorem}
\label{pml1}
Soit $f: \mathcal{PML}(S)\to \mathcal{PML}(S)$  un homéomorphism tel que pour tout  $\lambda$ et $\mu$ dans $\mathcal{PML}(S)$, on a l'équivalence $i(\lambda, \mu)=0 \Leftrightarrow i(f(\lambda), f(\mu))=0$.
Alors, $f$ est induit par un homéomorphisme de $S$.
\end{theorem}

\`A partir des théorèmes \ref{ml1} et \ref{pml1}, on démontre deux autres résultats qui donnent une caractérisation des groupes d'automorphismes de $\mathcal{ML}(S)$ et soit $\mathcal{PML}(S)$ préservant certaines structures. Avant de les énoncer on introduit quelques notations.

Soit   $\mathrm{Mod}(S)$ le groupe modulaire de $S$, c'est-à-dire le groupe des classes d'isotopie d'homéomorphismes de $S$ préservant l'orientation, et $\mathrm{Mod}^*(S)$ le groupe modulaire  étendu de $S$, c'est-à-dire le groupe des classes d'isotopie d'homéomorphismes quelconques de $S$.

\'Etant donné un homéomorphisme $f:\mathcal{ML}(S)\to \mathcal{ML}(S)$, on dit que $f$ préserve le nombre d'intersection si pour toutes laminations mesur\'ees $\lambda$ et $\mu$ de $\mathcal{ML}(S)$, on a  $i(\lambda, \mu)=  i(f(\lambda), f(\mu))$.

On désigne par $\mathrm{Aut}(\mathcal{ML}(S))$ le groupe d'automorphismes de $\mathcal{ML}(S)$ préservant le nombre d'intersection. On a alors un homomorphisme naturel  $\mathrm{Mod}^*(S)\to  \mathrm{Aut}(\mathcal{ML}(S))$.

On a alors :
\begin{theorem}
\label{ml} 

Pour $g\geq 3$, l'homomorphisme naturel  $\mathrm{Mod}^*(S)\to  \mathrm{Aut}(\mathcal{ML}(S))$ est un isomorphisme. Pour $g=2$, cet homomorphisme est surjectif et son noyau est le groupe à deux éléments $\mathbb{Z}/2$, engendré par l'involution hyperelliptique de $S$.
\end{theorem}

 Même si le nombre d'intersection entre deux éléments $\lambda$ et $\mu$ de $\mathcal{PML}(S)$ n'est pas défini, la relation $i(\lambda, \mu)=0$ a un sens.

On notera par $\mathrm{Aut}(\mathcal{PML}(S))$ le groupe d'homéomorphismes $f$ de $\mathcal{PML}(S)$ satisfaisant la propriété suivante :
pour tout $\lambda$ et $\mu$ dans $\mathcal{PML}$, on a l'équivalence $i(\lambda, \mu)=0 \Leftrightarrow i(f(\lambda), f(\mu))=0$.

On a un homomorphisme naturel  $\mathrm{Mod}^*(S)\to  \mathrm{Aut}(\mathcal{PML}(S))$.

On démontrera aussi le théorème suivant :

\begin{theorem}
\label{pml} 
Pour $g\geq 3$,  l'homomorphisme naturel  $\mathrm{Mod}^*(S)\to  \mathrm{Aut}(\mathcal{PML}(S))$ est un isomorphisme. Pour $g=2$, cet homomorphisme est surjectif et son noyau est $\mathbb{Z}/2$, engendré par l'involution hyperelliptique de $S$.
\end{theorem}

 On démontrera d'abord the théorème \ref{pml1}, ensuite le théorème \ref{ml1}, et enfin les théorèmes \ref{ml} et \ref{pml}.
 
On commence par établir quelques lemmes.

\'Etant donné un élément $\lambda$ de $\mathcal{PML}(S)$, on lui associe l'espace nul défini par
\[N(\lambda)=\{\mu \in \mathcal{PML}(S) \mid i(\mu, \lambda)=0\}.\]

\'Etant donné un hom\'eomorphisme $f : \mathcal{PML}(S) \to \mathcal{PML}(S)$ satisfaisant l'hypothèse du théorème \ref{pml1}, on a $N(f(\lambda))=f(N(\lambda))$.

 On notera $|\lambda|\subset S$ le support de $\lambda$.

La propri\'et\'e suivante est immédiate :

\begin{quote}
(*) si $|\lambda|$ est contenu en $|\lambda'|$,  on a $N(\lambda') \subset N(\lambda)$.
\end{quote}

Par contre, $N(\lambda') \subset N(\lambda)$ n'implique pas n\'ecessairement $|\lambda| \subset |\lambda'|$. 

\'Etant donnée une lamination g\'eod\'esique mesurée $l$, on appelle compl\'etion de $l$ la lamination g\'eod\'esique mesurée obtenue à partir de $l$ en ajoutant tous les bord des surfaces de support des composantes minimales de $l$. Ici, une surface d'une composante minimale est la plus petite sous-surface à bord géodésique de $S$ contenant cette composante.

On a alors le lemme suivant :
\begin{lemma}
\label{compl}

Si $N(\lambda') \subset N(\lambda)$, la compl\'etion de $|\lambda|$ est contenue dans celle de $|\lambda'|$.
\end{lemma}
\begin{proof}
Supposons que $N(\lambda') \subset N(\lambda)$. Comme $\lambda'$ est contenue dans $N(\lambda') \subset N(\lambda)$, il n'y a pas de composante de $\lambda$ qui rencontre $\lambda'$ transversalement.
Il s'ensuit que la compl\'etion de $|\lambda|$ ne rencontre pas non plus celle de $|\lambda'|$ transversalement.

Il suffit de prouver qu'il n'y a pas de composante minimale de $|\lambda|$ disjointe de la compl\'etion de $|\lambda'|$.
Supposons qu'il y en ait une, $l$.
Alors on peut trouver une g\'eod\'esique simple ferm\'ee disjointe de $|\lambda'|$ qui rencontre  $l$ transversalement, ce qui contredit l'hypoth\`ese que $N(\lambda') \subset N(\lambda)$.
\end{proof}

%

Pour abréger, on appellera \og courbe\fg une géodésique simple ferm\'ee de $S$. 

Dans ce qui suit, quand on parle de dimension d'un sous-espace de $\mathcal{ML}(S)$ ou de $\mathcal{PML}(S)$, il s'agit de dimension topologique, au sens de la structure de variété de ces espaces. (On peut se référer ici à la structure linéaire par morceaux, respectivement linéaire-projective par morceaux, de ces espaces.) Il sera sous-entendu que ces espaces sont des sous-variétés et que leur dimension est bien définie.

On démontre facilement le lemme suivant.
\begin{lemma}
\label{scc}
$\lambda$ est une courbe si et seulement si $\dim N(\lambda)=6g-8$.
\end{lemma}
On en déduit que pour tout $n\geq 1$, $f$ pr\'eserve l'ensemble des courbes.

On donne maintenant une caractérisation des courbes multiples pond\'er\'ees de $S$ qui permettra  de montrer que $f$ préserve l'ensemble des courbes  multiples pond\'er\'ees à nombre de composantes fixé.
\begin{lemma}\label{lem:courbes}
Pour tout $n\geq 1$, $\lambda \in \mathcal{PML}(S)$ est une courbe multiple pond\'er\'ee à $n$ composantes si  et seulement si les deux propriétés suivantes sont satisfaites :
\begin{enumerate}
\item
$\dim N(\lambda)=6g-7-n$ ;
\item
 il existe une suite $\lambda_n=\lambda, \lambda_{n-1}, \dots, \lambda_{1}$ telle que pour tout $j=1,\ldots,n$, $N(\lambda_{j+1}) \subset N(\lambda_j)$ et $\dim N(\lambda_j)=6g-7-j$.
\end{enumerate}
\end{lemma}
\begin{proof}
La partie \og seulement si\fg  découle  imm\'ediatement du lemme \ref{scc} et de  (*).

Pour d\'emontrer la partie \og si\fg, supposons que $\lambda_{j+1}$ soit une courbe multiple et que $N(\lambda_{j+1}) \subset N(\lambda_j)$.
Par le lemme \ref{compl}, on sait que $|\lambda_j|$ contient $|\lambda_{j+1}|$.
(Notons que la compl\'etion ne change pas $|\lambda_{j+1}|$, comme $\lambda_{j+1}$ est une courbe multiple.)
Si $|\lambda_j| \setminus |\lambda_{j+1}|$ n'est pas une courbe, alors $\dim N(\lambda_j) < \dim N(\lambda_{j+1})-1$, ce qui contredit l'hypoth\`ese.
\end{proof}

\begin{proof}[D\'emonstration du théorème \ref{pml1}]

Par le lemme \ref{lem:courbes}, l'homéomorphisme $f$ pr\'eserve l'ensemble des courbes multiples à nombre de composantes donné, et l'inclusion entre deux courbes multiples. Ainsi, $f$ induit un automorphisme du complexe de courbes $\mathcal{C}(S)$ de $S$. Par le théorème d'Ivanov \cite{Iv}, cet automorphisme est induit par un homéomorphisme $g$ de $S$ qui induit la m\^eme bijection sur $\mathcal C(S)$ que $f$.
Comme $\mathcal C(S)$ est dense dans $\mathcal{PML}(S)$,  la continuit\'e de $f$ et $g_*$ implique que $f=g_*$. 
\end{proof}

Passons maintenant à la démonstration du théorème \ref{ml1}. 
Pour une lamination mesurée $\lambda$, on définit son espace nul par $$\mathcal N(\lambda)=\{\mu \in \mathcal{ML}(S) \mid i(\lambda, \mu)=0\}.$$
De la même manière que le lemme \ref{scc}, on démontre que $\dim \mathcal N(\lambda)=6g-7$ si et seulement si $\lambda$ est une courbe pondérée.
De plus, l'espace nulle ne change pas même si l'on change le poids sur la courbe.
Il s'ensuit que pour toute courbe  $c$, l'homéomorphisme $f$ envoie le rayon $\mathbb{R}_+ c$ dans $\mathcal{ML}(S)$ sur le rayon $\mathbb{R}_+ f(c)$.
Comme les courbes  pondérées sont denses dans $\mathcal{ML}(S)$, on voit que $f$ envoie le rayon $\mathbb{R}_+ \lambda$ sur le rayon $\mathbb{R}_+ f(\lambda)$ pour toute lamination mesurée  $\lambda$.
Donc, l'hom\'eomorphisme $f$ du théorème \ref{ml1} induit un hom\'eomorphisme $\bar f:\mathcal{PML}(S)\to \mathcal{PML}(S)$ qui satisfait l'hypoth\`ese du théorème  \ref{pml1}.
Il existe alors, par le théorème \ref{pml1}, un automorphisme $g$ de $S$ tel que $g_*=\bar f$ sur $\mathcal{PML}(S)$.

\begin{lemma}
\label{add}
Soient $c_1, c_2$ deux courbes pond\'er\'ees disjointes.
Alors, $f(c_1 \cup c_2)=f(c_1)\cup f(c_2)$.
\end{lemma}
\begin{proof}
On a $i(., c_1 \cup c_2)=i(., c_1)+i(., c_2)$.
D'autre part,  si $i(.,\lambda)=i(., c_1)+i(., c_2)$ pour $c_1, c_2$ disjointes  alors on a $\lambda=c_1 \cup c_2$.
Par le théorème  \ref{pml1}, on sait que les deux courbes pondérées $f(c_1)$ et $f(c_2)$ sont disjointes.
Comme $i(., f(c_1 \cup c_2))=i(f^{-1}(.), c_1 \cup c_2)= i(f^{-1}(.), c_1)+i(f^{-1}(.), c_2)=i(., f(c_1))+i(., f(c_2))$, on a $f(c_1 \cup c_2)=f(c_1)\cup f(c_2)$.
\end{proof}

\'Etant donnée une courbe $c$, on pose $f(c)=m_c g(c)$ o\`u $g$ est un automorphisme obtenu par le théorème \ref{pml1} et $m_c$ un poids.
Pour compl\'eter la d\'emonstration du théorème \ref{ml1}, il suffit de démontrer le lemme suivant :

\begin{lemma}
$m_c=1$ pour toute courbe $c$.
\end{lemma}
\begin{proof}
Montrons que $m_c$ ne d\'epend pas de $c$.
D'abord on consid\`ere deux courbes disjointes $c$ et $d$.
Comme $f=g_*$ sur $\mathcal{PML}(S)$, on a $f(c\cup d)=m(g(c)\cup g(d))$.
Par le lemme \ref{add}, on a $f(c\cup d)=f(c)\cup f(d)$, et par d\'efinition $f(c)=m_cg(c)$ et $f(d)=m_d g(d)$.
Donc on a $m_c=m=m_d$.

\'Etant données deux courbes quelconques $c$ et $d$, il existe une  suite de courbes $c=c_1, \dots, c_n=d$ telles que $c_j$ et $c_{j+1}$ sont disjointes, par la connexit\'e du complexe des courbes.
Donc on a $m_c=m_1=m_2 = \dots =m_n=m_d$.

Ainsi, $m_c$ ne d\'epend pas de $c$, et de là il découle facilement que $m_c=1$.
\end{proof}

\begin{proof}[Démonstration des théorèmes \ref{ml} et \ref{pml}]

Par les théorèmes \ref{ml1} et \ref{pml1}, les homomorphismes    $\mathrm{Mod}^*(S)\to  \mathrm{Aut}(\mathcal{ML}(S))$ et $\mathrm{Mod}^*(S)\to  \mathrm{Aut}(\mathcal{PML}(S))$  sont surjectifs.
 Pour $g\geq3$, ils sont tous les deux injectifs car si deux éléments de $\mathrm{Mod}^*(S)$ ont la même action sur $\mathcal{ML}(S)$ ou sur $\mathcal{PML}(S)$, ils induisent le même automorphisme de $\mathcal{C}(S)$, et l'on sait par \cite{Iv} que l'homomorphisme  $\mathrm{Mod}^*(S)\to  \mathrm{Aut}(\mathcal{C}(S))$ est injectif. Pour $g=2$, le noyau de chacun des homomorphismes  $\mathrm{Mod}^*(S)\to  \mathrm{Aut}(\mathcal{ML}(S))$ et $\mathrm{Mod}^*(S)\to  \mathrm{Aut}(\mathcal{PML}(S))$   est $\mathbb{Z}/2$, ce qui découle aussi de \cite{Iv}, 
o\`u il est montré que dans ce cas l'homomorphisme naturel  $\mathrm{Mod}^*(S)\to  \mathrm{Aut}(\mathcal{C}(S))$ est 
 surjectif et son noyau est $\mathbb{Z}/2$, engendré par l'involution hyperelliptique de $S$.
 
\end{proof}
  
 \begin{remark} Les résultats de cet article peuvent être considérés comme des variations sur des résultats de Feng Luo dans \cite{Lu2}, même si les énoncés et les démonstrations sont diff\'erents. En particulier, Luo étudie sur $\mathcal{ML}(S)$ des automorphismes d'espaces de fonctions associées à des courbes, ainsi que des espaces de zéros de ces fonctions, et dans ce cas l'action induite sur le complexe de courbes $\mathcal{C}(S)$ est immédiate.
 \end{remark}

\end{document}